\UseRawInputEncoding
\documentclass[12pt, twoside, reqno]{article}

\usepackage{amssymb}
\usepackage{amsmath}
\usepackage{amsthm}
\usepackage{color}
\usepackage[labelfont=footnotesize,font=footnotesize]{caption}
\usepackage[colorlinks, citecolor=red]{hyperref}
\usepackage{cite}
\usepackage{environ}
\usepackage{xparse}
\usepackage{mathrsfs}
\usepackage{verbatim}
\usepackage{indentfirst}

\usepackage{txfonts}

\usepackage{setspace}

\usepackage{fancyhdr}
\def\rr{{\mathbb R}}
\def\rn{{\mathbb{R}^n}}

\def\zz{{\mathbb Z}}

\def\nn{{\mathbb N}}

\def\fz{\infty }
\def\az{\alpha}

\def\lz{\lambda}

\def\lf{\left}
\def\r{\right}

\def\supp{\mathop\mathrm{\,supp\,}}

\def\XXint#1#2#3{{\setbox0=\hbox{$#1{#2#3}{\int}$ }
\vcenter{\hbox{$#2#3$ }}\kern-.6\wd0}}

\DeclareMathOperator{\esssup}{ess\,sup}

\def\({\left(}
\def \){ \right)}

\def\lz{{\lambda}}

\newtheorem{theorem}{Theorem}[section]
\newtheorem{corollary}{Corollary}[section]
\newtheorem{lemma}{Lemma}[section]

\newtheorem{definition}{Definition}[section]

\newtheorem{proposition}{Proposition}[section]
\newtheorem{remark}{Remark}[section]

\renewcommand{\appendix}{\par
   \setcounter{section}{0}%
   \setcounter{subsection}{0}%
   \setcounter{subsubsection}{0}%
   \gdef\thesection{\@Alph\c@section}%
   \gdef\thesubsection{\@Alph\c@section.\@arabic\c@subsection}%
   \gdef\theHsection{\@Alph\c@section.}%
   \gdef\theHsubsection{\@Alph\c@section.\@arabic\c@subsection}%
   \csname appendixmore\endcsname
 }

\numberwithin{equation}{section}

\begin{document}
\title{\bf\Large Herz-slice spaces and applications
\footnotetext{\hspace{-0.35cm} 2020 \emph{\it
Mathematics Subject Classification}. Primary 42B35;
Secondary 35R35, 46E30, 42B25.
\endgraf{\emph {Key words}}. Herz-slice spaces; weak spaces; dual spaces; the Hardy--Littlewood maximal operator.
\endgraf This project is supported
by the National Natural Science Foundation of China (Grant No. 12061069),
the Natural Science Foundation Project of Chongqing, China (Grant No. cstc2021jcyj-msxmX0705).}}
\author{Yuan Lu, Jiang Zhou\,\footnote{Corresponding author
e-mail: \texttt{Zhoujiang@xju.edu.cn}.
}, ~Songbai Wang
\\[.5cm]}
\date{}
\maketitle

\vspace{-0.7cm}

\begin{center}

\begin{minipage}{13cm}
{\small {\bf Abstract}\quad Let $\alpha\in\rr$, $t\in(0,\infty)$, $p\in(0,\infty]$, $r\in(1,\fz)$ and $q\in[1,\fz]$.
We introduce the homogeneous Herz-slice space $(\dot KE_{q,r}^{\alpha,p})_t(\rn)$, the non-homogeneous Herz-slice space
$(KE_{q,r}^{\alpha,p})_t(\rn)$
and show some properties of them. As an application, the bounds for
the Hardy--Littlewood maximal operator on these spaces is considered.

}
\end{minipage}
\end{center}

%
%

\section{Introduction\label{s1}}
The homogeneous Herz-slice space $(\dot KE_{q,r}^{\alpha,p})_t(\rn)$ and the non-homogeneous Herz-slice space
$(\dot KE_{q,r}^{\alpha,p})_t(\rn)$ studied in this paper are associated with classical Herz spaces and
the slice space. We will first give a short history of the slice space.

In 1926, Wiener \cite{WN} first introduced amalgam spaces to formulate his generalized harmonic analysis.
In general,
for $p,\,q\in(0,\infty)$, the amalgam space $(L^{p},\ell^{q})(\mathbb R)$ is defined by
$$
(L^{p},\ell^{q})(\mathbb R):
=\lf\{f\in L_{\mathrm{loc}}^p(\rr):\,\lf[\sum_{n\in\mathbb Z}\lf\|f\mathbf 1_{[n,n+1)}\r\|_{L^p(\rr)}^q\r]^\frac1q<\infty\r\}.
$$
But the first systematic study of these spaces was undertaken by Holland \cite{HF} in 1975.
In recognition of Wiener's first use of amalgams, Feichtinger initially called these spaces
"Wiener-type spaces", and then,
adopted the
name "Wiener amalgam spaces",
that's also the most general definition of the amalgam space so far which provided by Feichtinger in the early 1980's
in a series of papers \cite{HGF1,HGF2,HGF3,F}.

In 2014, Auscher and Mourgoglou \cite{AM} introduced  particular cases of Wiener amalgam spaces \cite{HGF3}, the slice space $E_t^p(\rn)$ (also $(E_2^p)_t(\rn)$),
to study the classification of weak solutions in the natural classes for the boundary value
problems of a t-independent elliptic system in the upper half-space.
Moreover, in 2017, Auscher and Prisuelos-Arribas \cite{AP} introduced a more general slice space $(E^q_r)_t(\rn)$, that is,
for $t\in(0,\fz)$, $r\in(1,\fz)$ and $q\in[1,\fz]$, the slice space $(E_r^q)_t(\rn)$ is defined by the set of all measurable functions $f$ such that
$$
\|f\|_{(E_r^q)_t(\rn)}:=\lf\|\lf(\frac1{|B(\cdot,t)|}\int_{B(\cdot,t)}|f(y)|^r\,dy\r)^\frac 1r\r\|_{L^q(\rn)}<\infty,
$$
with the usual modification when $q=\infty$, in fact, when $r=2$, $q=p$, that is $E_t^p(\rn)$ in \cite{AM},
furthermore, the authors also show the boundedness of some classical operators over these spaces.
For more studies and developments about the slice space we may consult \cite{Ho,ZYYW}
and the references therein.

In 1968, Herz \cite{Herz} introduced Herz spaces in the study of absolutely convergent Fourier transforms.
For $\alpha\in\rr$, $p,\,q\in(0,\infty]$, the homogeneous Herz space
$(\dot K_q^{\alpha,p})(\rn)$ is defined by
\begin{equation}\label{DefH}
(\dot K_q^{\alpha,p})(\rn):=\lf\{f\in L_{\mathrm{loc}}^q(\rn\setminus\{0\}):\lf[\sum_{k=-\infty}^\infty2^{k\alpha p}\lf\|f\mathbf1_{S_k}\r\|_{L^q}^p\r]^\frac1p<\infty\r\},
\end{equation}
and the non-homogeneous Herz space
$(K_q^{\alpha,p})(\rn)$ is defined by
\begin{equation}\label{DefnH}
(K_q^{\alpha,p})(\rn):=\lf\{f\in L_{\mathrm{loc}}^q(\rn):\lf[\sum_{k=0}^\infty2^{k\alpha p}\lf\|f\mathbf1_{S_k}\r\|_{L^q}^p\r]^\frac1p<\infty\r\},
\end{equation}
but the first study of these spaces was undertaken by Beurling \cite{BA}.
In the 1990's, Lu and Yang \cite{LY} introduced the preceeding of the homogeneous Herz space and the
non-homogeneous Herz space with general indices and established the block
decomposition of the Herz space, from which they showed many
properties of these spaces.
In recent years, a series of papers have paid attention to the study of the Herz-type space,
we refer to \cite{HY,LSZ,LX,WHB,WMQ} and so on.

In this paper, the homogeneous Herz-slice space and the non-homogeneous Herz-slice space are introduced, we further show some properties over these spaces, such as the relationship between the homogeneous Herz-slice space and the non-homogeneous Herz-slice space, their dual spaces, a decomposition characterization of these spaces. Moreover, the bounds for the
Hardy--Littlewood maximal operator over the homogeneous Herz-slice space and the non-homogeneous Herz-slice space is obtained.

This paper is organized as follows. The definition of Herz-slice spaces, weak Herz-slice spaces and their main remarks will be given in Section \ref{m}. In Section \ref{H}, we show main properties of Herz-slice spaces.
In Section \ref{B}, we obtain the dual of Herz-slice spaces and a decomposition characterization of Herz-slice spaces.
In the finial section, the boundedness of the Hardy--Littlewood maximal function is given on the homogeneous Herz-slice space $(\dot KE_{q,r}^{\alpha,p})_t(\rn)$ and the non-homogeneous Herz-slice space $(KE_{q,r}^{\alpha,p})_t(\rn)$.

Finally, we make some conventions on notation. Let $B_k=B(0,2^k)=\{x\in\rn:\,|x|\leq2^k\}$ and $S_k:=B_k\setminus B_{k-1}$ for any $k\in\zz$.
Denote $\mathbf 1_{k}=\mathbf 1_{S_k}$ for $k\in\zz$, and $\mathbf 1_{S_0}=\mathbf 1_{B_0}$,
where $\mathbf 1_{k}$ is the characteristic function of $S_k$.
We write $A\lesssim B$ to mean that there exists a positive constant $C$ such that $A\leq CB$.
$A\sim B$ denotes that $A\lesssim B$ and $B\lesssim A$.
Throughout this paper, the letter $C$ will be used for positive
constants independent of relevant variables that may change from
one occurrence to another.
\section{Main definitions}\label{m}
To state the definition of the Herz-slice space,
we recall some necessary definitions.

For $p\in(0,\infty)$, the Lebesgue space $L^p(\mathbb R^n)$
is defined as the set of all measurable functions $f$ on $\mathbb R^n$ such that
$$
\|f\|_{L^p(\mathbb R^n)}:=\lf[\int_{\mathbb R^n}|f(x)|^p\,dx\r]^\frac1p<\infty.
$$
The weak Lebesgue space $L^{p,\infty}(\mathbb R^n)$ is defined as the set of all measurable functions $f$ on $\mathbb R^n$ such that
$$
\|f\|_{L^{p,\infty}(\mathbb R^n)}:=\sup_{\alpha>0}\alpha|\{x\in\mathbb R^n:\,|f(x)|>\alpha\}|^\frac1p<\infty.
$$
For $p=\infty$,
$$
\|f\|_{L^{\infty}(\mathbb R^n)}:=\underset{x\in\mathbb R^n}{\esssup}|f(x)|<\infty.
$$
\begin{definition}\label{Def}
Let $\alpha\in\rr$, $t\in(0,\infty)$, $p\in(0,\infty]$, $r\in(1,\fz)$ and $q\in[1,\fz]$. The homogeneous Herz-slice space
$(\dot KE_{q,r}^{\alpha,p})_t(\rn)$ is defined as the set of $f\in L_{\mathrm{loc}}^r(\rn\setminus\{0\})$ such that
$$
\lf\|f\r\|_{(\dot KE_{q,r}^{\alpha,p})_t(\rn)}:=\lf[\sum_{k=-\infty}^\infty2^{k\alpha p}\lf\|f\mathbf1_{S_k}\r\|_{(E_r^q)_t(\rn)}^p\r]^\frac1p<\infty,
$$
with the usual modification when $q,\,p=\infty$.

The non-homogeneous Herz-slice space $(KE_{q,r}^{\alpha,p})_t(\rn)$ is defined as the set of all measurable functions
$f$ such that
$$
\lf\|f\r\|_{(KE_{q,r}^{\alpha,p})_t(\rn)}:=\lf[\sum_{k=0}^\infty2^{k\alpha p}\lf\|f\mathbf1_{S_k}\r\|_{(E_r^q)_t(\rn)}^p\r]^\frac1p<\infty,
$$
with the usual modification when $q,\,p=\infty$.
\end{definition}
\begin{remark}
Let $\alpha\in\rr$, $t\in(0,\infty)$, $p\in(0,\infty]$, $r\in(1,\fz)$ and $q\in[1,\fz]$.
\begin{itemize}
\item[(1)] For any $s\in(0,\infty)$, we see
\begin{equation}\label{Es}
\lf\||f|^s\r\|_{(\dot KE_{q,r}^{\alpha,p})_t(\rn)}^\frac1s=\|f\|_{(\dot KE_{sq,sr}^{\alpha/s,sp})_t(\rn)};
\end{equation}
\item[(2)] $(\dot KE_{q,r}^{0,q})_t(\rn)=E^q_r(\rn)$, and for $\az\in\rr$, $(\dot KE_{q,r}^{\az,q})_t(\rn)
=(L^r,L^q_w)_t(\rn)$ with $w$ is a power weight $|x|^{\az q}$ \cite{LWZ};
\item[(3)] For $\az\in\rr$, if $r=q$, $(\dot KE_{q,r}^{\az,p})_t(\rn)=(\dot K_q^{\alpha,p})(\rn)$, and
$(KE_{q,r}^{\az,p})_t(\rn)=(K_q^{\alpha,p})(\rn)$;
\item[(4)] $(\dot KE_{p,p}^{0,p})_t(\rn)=L^p(\rn)$, and for $\az\in\rr$,
$(\dot KE_{p,p}^{\az,p})_t(\rn)=L^p_w(\rn)$ with $w=|x|^{\az q}$.
\end{itemize}
\end{remark}
\begin{definition}
Let $t\in(0,\fz)$, $r\in(1,\fz)$ and $q\in[1,\fz]$, the weak slice space
$W(E_r^q)_t(\rn)$ is defined as the set of all measurable functions $f$ such that
$$
\|f\|_{W(E_r^q)_t(\rn)}:=\sup\limits_{\lz>0}\lz\lf\|\mathbf 1_{\{x\in \rn:~|f(x)|>\lambda\}}\r\|_{(E_r^q)_t(\rn)}<\infty.
$$
In fact, the weak slice space is the weak amalgam space $W(L^r_w,L^q_v)_t(\rn)$ with $w=v=1$ in \cite{LWZ}.
\end{definition}
\begin{definition}
Let $\alpha\in\rr$, $t\in(0,\infty)$, $p\in(0,\infty]$, $r\in(1,\fz)$ and $q\in[1,\infty]$. The homogeneous weak Herz-slice space
$W(\dot KE_{q,r}^{\alpha,p})_t(\rn)$ is defined as the set of all measurable functions $f$ such that
$$
\|f\|_{W(\dot KE_{q,r}^{\alpha,p})_t(\rn)}:=\sup\limits_{\lz>0}\lz\lf(\sum_{k\in\zz}2^{k\az p}\lf\|\mathbf 1_{\lf\{x\in S_k:~|f(x)|>\lambda\r\}}\r\|^p_{(E_r^q)_t(\rn)}\r)^\frac1p<\infty.
$$
The non-homogeneous weak Herz-slice space $W(KE_{q,r}^{\alpha,p})_t(\rn)$ is defined as the set of all measurable functions $f$ such that
$$
\|f\|_{W(KE_{q,r}^{\alpha,p})_t(\rn)}:=\sup\limits_{\lz>0}\lz\lf(\sum_{k=0}^\fz2^{k\az p}\lf\|\mathbf 1_{\lf\{x\in S_k:~|f(x)|>\lambda\r\}}\r\|^p_{(E_r^q)_t(\rn)}\r)^\frac1p<\infty.
$$
\end{definition}
It is obvious that $W(\dot KE_{r,p}^{0,p})_t(\rn)=W(KE_{r,p}^{0,p})_t(\rn)=W(E_r^p)_t(\rn)$, and for $r=q$, $W(\dot KE_{r,q}^{\az,p})_t(\rn)$ and $W(KE_{r,q}^{\az,p})_t(\rn)$ are the homogeneous weak Herz space $W(\dot K_q^{\alpha,p})(\rn)$ and the non-homogeneous weak Herz space $W(K_q^{\alpha,p})(\rn)$ in \cite{LYH}, respectively.
\section{Properties of Herz-slice spaces}\label{H}
In this section, some main properties of the homogeneous Herz-slice space $(\dot KE_{q,r}^{\alpha,p})_t(\rn)$ and the non-homogeneous Herz-slice space
$(KE_{q,r}^{\alpha,p})_t(\rn)$ will be given.
\begin{proposition}
$(\dot KE_{q,r}^{\alpha,p})_t(\rn)$ and $(KE_{q,r}^{\alpha,p})_t(\rn)$ are quasi-Banach function spaces, and if $p,q,r\ge1$, $(\dot KE_{q,r}^{\alpha,p})_t(\rn)$ and $(KE_{q,r}^{\alpha,p})_t(\rn)$ are Banach spaces;
\end{proposition}
Since the proof is analogue to that of the classical Herz space, and the details are omitted.
\begin{proposition}
Let $\az\in\rr$, $p\in(0,\fz]$, $t\in(0,\infty)$, $r\in(1,\fz)$ and $q\in[1,\fz]$. The following inclusions are valid:
\begin{enumerate}
  \item[(1)] if $p_1\leq p_2$, then $(\dot KE_{q,r}^{\alpha,p_1})_t(\rn)\subset
  (\dot KE_{q,r}^{\alpha,p_2})_t(\rn)$ and $ (KE_{q,r}^{\alpha,p_1})_t(\rn)
  \subset(KE_{q,r}^{\alpha,p_2})_t(\rn)$;
  \item[(2)] if $\az_2\leq \az_1$, then $(KE_{q,r}^{\alpha_1,p})_t(\rn)
  \subset(KE_{q,r}^{\alpha_2,p})_t(\rn)$.
\end{enumerate}
\end{proposition}
\begin{proof}
This proposition can be proved by fairly simple computation. In fact, (1)
is a consequence of the inequality in \cite{JK}
\begin{equation}\label{cp}
\lf(\sum\limits_{k=1}^\fz|a_k|\r)^r\leq\sum\limits_{k=1}^\infty|a_k|^r, ~~~~if~ 0<r<1.
\end{equation}
(2) can be deduced from the H\"older inequality directly.
\end{proof}
\begin{proposition}\label{L}
Let $0<\az,t<\infty$, $p\in(0,\fz]$, $r\in(1,\fz)$ and $q\in[1,\fz]$. Then
$$
(KE_{q,r}^{\alpha,p})_t(\rn)=(\dot KE_{q,r}^{\alpha,p})_t(\rn)\cap (E^q_r)_t(\rn),
$$
and for $f\in(\dot KE_{q,r}^{\alpha,p})_t(\rn)\cap (E^q_r)_t(\rn)$,
$$
\|f\|_{(KE_{q,r}^{\alpha,p})_t(\rn)}\sim\|f\|_{(\dot KE_{q,r}^{\alpha,p})_t(\rn)}+\|f\|_{(E^q_r)_t(\rn)}.
$$
\end{proposition}
\begin{proof}
If $f\in(\dot KE_{q,r}^{\alpha,p})_t(\rn)\cap (E^q_r)_t(\rn)$, then
\begin{align*}
\lf\|f\r\|^p_{(KE_{q,r}^{\alpha,p})_t(\rn)}
   &=\lf\|f\mathbf1_{S_0}\r\|_{(E^q_r)_t(\rn)}+\sum\limits_{k=1}^\infty2^{k\az p}
\lf\|f\mathbf1_{S_k}\r\|_{(E^q_r)_t(\rn)}
   \leq \|f\|^p_{(E^q_r)_t(\rn)}+\|f\|^p_{(\dot KE_{q,r}^{\alpha,p})_t(\rn)}.
\end{align*}
We claim that
$$
\sum\limits_{k=-\infty}^02^{k\az p}\lf\|f\mathbf1_{S_k}\r\|^p_{(E^q_r)_t(\rn)}
\leq C\lf\|f\mathbf1_{S_0}\r\|^p_{(E^q_r)_t(\rn)}.
$$
If $1\leq q<p$, by (\ref{cp}), then
\begin{align*}
\sum\limits_{k=-\infty}^02^{k\az p}\lf\|f\mathbf1_{S_k}\r\|^p_{(E^q_r)_t(\rn)}
&\leq \lf(\sum\limits_{k=-\infty}^02^{k\az q}\lf\|f\mathbf1_{S_k}\r\|^q_{(E^q_r)_t(\rn)}\r)^{p/q}
\leq C\lf\|f\mathbf1_{S_0}\r\|^p_{(E^q_r)_t(\rn)}.
\end{align*}
H\"older's inequality yields
$$
\sum\limits_{k=-\infty}^02^{k\az p}\lf\|f\mathbf1_{S_k}\r\|^p_{(E^q_r)_t(\rn)}
\leq \lf(\sum\limits_{k=-\infty}^02^{k\az p(q/p)'}\r)^{1/(q/p)'}\lf(\sum\limits_{k=-\infty}^0
\lf\|f\mathbf1_{S_k}\r\|^q_{(E^q_r)_t(\rn)}\r)^{p/q}
\leq C\lf\|f\mathbf1_{S_0}\r\|^p_{(E^q_r)_t(\rn)}.
$$
This proves our claim.
From the fact that $\lf\|f\mathbf1_{S_0}\r\|^p_{(E^q_r)_t(\rn)}\leq \lf\|f\r\|_{(KE_{q,r}^{\alpha,p})_t(\rn)}$,
to obtain $f\in(KE_{q,r}^{\alpha,p})_t(\rn)$ deduces $f\in(\dot KE_{q,r}^{\alpha,p})_t(\rn)\cap (E^q_r)_t(\rn)$,
it suffices to prove that
$$
\lf\|f\mathbf1_{S_0^c}\r\|_{(E^q_r)_t(\rn)}\leq \lf\|f\r\|_{(KE_{q,r}^{\alpha,p})_t(\rn)}.
$$
For $1\leq q<p$, it follows from the H\"older inequality and (\ref{Es}) that
\begin{align*}
\lf\|f\mathbf1_{S_0^c}\r\|_{(E^q_r)_t(\rn)}&=\sum\limits_{k=1}^\fz\lf\|f\mathbf1_{S_k}\r\|_{(E^q_r)_t(\rn)}\\
    &\leq \lf(\sum\limits_{k=1}^\infty2^{-k\az q(p/q)'}\r)^{1/(p/q)'}\lf(\sum\limits_{k=1}^\fz2^{k\az p}
\lf\|f\mathbf1_{S_k}\r\|^{p/q}_{(E^q_r)_t(\rn)}\r)^{q/p}\\
    &\leq C\|f\|^q_{(KE_{q,r}^{\alpha,p})_t(\rn)}.
\end{align*}
For $0<p\leq q$, using (\ref{cp}) one can see
$$
\lf\|f\mathbf1_{S_0^c}\r\|^q_{(E^q_r)_t(\rn)}
\leq\lf(\sum\limits_{k=1}^\fz\lf\|f\mathbf1_{S_k}\r\|^p_{(E^q_r)_t(\rn)}\r)^{q/p}
\leq \lf\|f\r\|^q_{(KE_{q,r}^{\az,p})_t(\rn)}.
$$
This completes the proof of proposition \ref{L}.
\end{proof}
\section{Block decompositions and dual spaces}\label{B}
This section is devoted to the decomposition characterizations and the dual spaces for Herz-slice spaces as given
in Definition \ref{Def}.
\begin{definition}
Let $0<\az,t<\infty$, $1<r<\infty$ and $1\leq q<\fz$.\\
(i) A function $a(x)$ on $\rn$ is said to be a central $(\az,q,r)$-block if
  \begin{enumerate}
    \item[(1)] $\supp(a)\subset B(0,R)$, for some $R>0$;
    \item[(2)] $\|a\|_{(E^q_r)_t(\rn)}\leq CR^{-\az}$.
  \end{enumerate}
(ii) A function $b(x)$ on $\rn$ is said to be a central $(\az,q,r)$-block of restrict type if
  \begin{enumerate}
    \item[(1)] $\supp(b)\subset B(0,R)$ for some $R\ge 1$;
    \item[(2)] $\|b\|_{(E^q_r)_t(\rn)}\leq CR^{-\az}$.
  \end{enumerate}
If $R=2^k$ for some $k\in\zz$, then the corresponding central block is called a dyadic central block.
\end{definition}
\begin{theorem}\label{Thdf}
Let $0<\az,\,t,\,p<\fz$, $1<r<\fz$ and $1\le q<\fz$. The following statements are equivalent:

(1) $f\in(\dot KE_{q,r}^{\alpha,p})_t(\rn)$.

(2) $f$ can be represented by
  \begin{equation}\label{df}
  f(x)=\sum\limits_{k\in\zz}\lz_kb_k(x),
  \end{equation}
where $\sum\limits_{k\in\zz}|\lz_k|^p<\fz$ and each $b_k$ is a dyadic central $(\az,q,r)$-block with support contained in $B_k$.
\end{theorem}
\begin{proof}
We suppose initially that $f\in(\dot KE_{q,r}^{\alpha,p})_t(\rn)$, write
\begin{align*}
f(x)&=\sum\limits_{k}f(x)\mathbf1_{S_k}(x)
    =\sum\limits_{k}|B_k|^{\az/n}\lf\|f\mathbf1_{S_k}\r\|_{(E^q_r)_t(\rn)}
\frac{f(x)\mathbf1_{S_k}(x)}{|B_k|^{\az/n}\lf\|f\mathbf1_{S_k}\r\|_{(E^q_r)_t(\rn)}}
    =\sum\limits_{k\in\zz}\lz_kb_k(x),
\end{align*}
where
$$
\lz_k=|B_k|^{\az/n}\lf\|f\mathbf1_{S_k}\r\|_{(E^q_r)_t(\rn)}
~~\text{and}~~
b_k(x)=\frac{f(x)\mathbf1_{S_k}(x)}{|B_k|^{\az/n}\lf\|f\mathbf1_{S_k}\r\|_{(E^q_r)_t(\rn)}}.
$$
Obviously, $\supp(b_k)\subset B_k$, and $\|b_k\|_{(E^q_r)_t(\rn)}=|B_k|^{-\az/n}$.
Thus, each $b_k$ is a dyadic central $(\az,q,r)$-block with the support $B_k$ and
$$
\sum\limits_{k}|\lz_k|^p=\sum\limits_{k}\lf(|B_k|^{\az/n}\lf\|f\mathbf1_{S_k}\r\|_{(E^q_r)_t(\rn)}\r)^p
=\sum\limits_{k}|B_k|^{\frac{\az p}n}\lf\|f\mathbf1_{S_k}\r\|^p_{(E^q_r)_t(\rn)}=
\|f\|^p_{(\dot KE_{q,r}^{\az,p})_t(\rn)}<\fz.
$$

$(2)\Rightarrow (1)$. Let $f(x)=\sum\limits_{k\in\zz}\lz_kb_k(x)$ be a decomposition of $f$ which satisfies the hypothesis (2) of Theorem \ref{Thdf}.  For each $j\in\zz$, it is readily to see that
\begin{equation}\label{edf}
\lf\|f\mathbf1_{S_j}\r\|_{(E^q_r)_t(\rn)}\leq \sum\limits_{k\ge j}|\lz_k|\lf\|b_k\r\|_{(E^q_r)_t(\rn)}.
\end{equation}
For $0<p\leq1$, from (\ref{edf}), it follows that
\begin{align*}
\|f\|^p_{(\dot KE_{q,r}^{\alpha,p})_t(\rn)}
&=\sum\limits_{k\in\zz}2^{k\alpha p}\lf\|f\mathbf1_{S_k}\r\|^p_{(E_r^q)_t(\rn)}
\leq \sum\limits_{k\in\zz}2^{k\alpha p}\lf(\sum\limits_{j\ge k}|\lz_k|^p\|b_j\|^p_{(E_r^q)_t(\rn)}\r)\\
&\leq \sum\limits_{k\in\zz}2^{k\alpha p}\lf(\sum\limits_{j\ge k}|\lz_j|^p2^{-\az jp}\r)
=\sum\limits_{k\in\zz}|\lz_k|^p\sum\limits_{j\ge k}2^{-\az (j-k)p}
\leq C\sum\limits_{k\in\zz}|\lz_k|^p<\fz.
\end{align*}

For the case of $1<p<\fz$, again by (\ref{edf}) and the H\"older inequality,
\begin{align*}
\lf\|f\mathbf1_{S_j}\r\|_{(E^q_r)_t(\rn)}&\leq
\sum\limits_{k\ge j}|\lz_k|\lf\|b_k\r\|^\frac12_{(E^q_r)_t(\rn)}\lf\|b_k\r\|^\frac12_{(E^q_r)_t(\rn)}\\
&\leq \lf(\sum\limits_{k\ge j}|\lz_k|^p\|b_k\|^{\frac p2}_{(E^q_r)_t(\rn)}\r)^\frac1p
\lf(\sum\limits_{k\ge j}\|b_k\|^{\frac{p'}2}_{(E^q_r)_t(\rn)}\r)^{\frac1{p'}}\\
&\leq \lf(\sum\limits_{k\ge j}|\lz_k|^p2^{-\frac {\az kp}2}\r)^\frac1p
\lf(\sum\limits_{k\ge j}2^{-\frac{\az kp'}2}\r)^{\frac1{p'}}.
\end{align*}
Therefore,
\begin{align*}
\lf\|f\r\|^p_{(\dot KE_{q,r}^{\alpha,p})_t(\rn)}
&\leq C\sum_{j\in\zz}2^{\az jp}
\lf(\sum\limits_{k\ge j}|\lz_k|^p2^{-\frac {\az kp}2}\r)^\frac1p
\lf(\sum\limits_{k\ge j}2^{-\frac{\az kp'}2}\r)^{\frac1{p'}}\\
&\leq C\sum_{j\in\zz}|\lz_k|^p\sum\limits_{j\leq k}2^{\az (j-k)p/2}
\leq C\sum_{k\in\zz}|\lz_k|^p<\fz.
\end{align*}
This leads to that $f\in(\dot KE_{q,r}^{\alpha,p})_t(\rn)$ and then completes the proof of Theorem \ref{Thdf}.
\end{proof}
\begin{remark}
From the proof of Theorem \ref{Thdf}, it's easy to see that
$$
\|f\|_{(\dot KE_{q,r}^{\alpha,p})_t(\rn)}\sim\lf(\sum_{k\in\zz}|\lz_k|^p\r)^\frac1p.
$$
\end{remark}
Similarly for the decompositional characterizations of the homogeneous Herz-slice space,
it shows that the decompositional over the non-homogeneous Herz-slice spaces as follows.
\begin{theorem}\label{Thndf}
Let $0<\az,\,t,\,p<\fz$, $1< r<\fz$ and $1\leq q<\fz$. The following statements are equivalent:

(1) $f\in(KE_{q,r}^{\alpha,p})_t(\rn)$.

(2) $f$ can be represented by
  \begin{equation}\label{ndf}
  f(x)=\sum\limits_{k=0}^\fz\lz_kb_k(x),
  \end{equation}
where $\sum\limits_{k=0}^\fz|\lz_k|^p<\fz$ and each $b_k$ is a dyadic central $(\az,q,r)$-block of restrict type
  with support contained in $B_k$.
Moreover,
$$
\|f\|_{(KE_{q,r}^{\alpha,p})_t(\rn)}\sim\lf(\sum_{k\ge 0}|\lz_k|^p\r)^\frac1p.
$$
\end{theorem}
For stating the dual spaces of Herz-slice spaces, we recall the H\"older inequality over the slice space in the following.
\begin{lemma}\label{Holder}\cite{Duiou}
Given $1\leq q\leq\infty$ and $1<r<\fz$,
$$
\|fg\|_{L^1(\mathbb R^n)}\leq\|f\|_{(E^q_r)_t(\rn)}\|g\|_{(E^{q'}_{r'})_t(\rn)},
$$
where $\frac{1}{q}+\frac{1}{q'}=\frac{1}{r}+\frac{1}{r'}=1$.
\end{lemma}

\begin{theorem}\label{KDuiO}
Let $\alpha\in\rr$, $t,\,p\in(0,\infty)$, $r\in(1,\infty)$ and $q\in[1,\infty)$. The dual space of
$(\dot KE_{q,r}^{\alpha,p})_t(\rn)$ is
$$
\left(\left(\dot KE_{q,r}^{\alpha,p}\right)_t(\rn)\right)^*=\begin{cases}
\left(\dot KE_{q',r'}^{-\alpha,p'}\right)_t(\rn),\quad 1<p<\infty;\\
\left(\dot KE_{q',r'}^{-\alpha,\infty}\right)_t(\rn),\quad 0<p\leq1.
\end{cases}
$$
\end{theorem}
\begin{proof}
We first prove it for the case $1<p<\infty$. Let $T\in (\dot KE_{q',r'}^{-\alpha,p'})_t(\rn)$.
For any $f\in(\dot KE_{q,r}^{\alpha,p})_t(\rn)$, we define
$$
(T,f):=\int_\rn T(x)f(x)\,dx=\sum_{k=-\infty}^\infty\int_{S_k}T(x)f(x)\,dx.
$$
By Lemma \ref{Holder}, we have
$$
|(T,f)|\leq\left[\sum_{k=-\infty}^\infty2^{-k\alpha p'}\|T\|_{(E^{q'}_{r'})_t(S_k)}^{p'}\right]^\frac1{p'}
\left[\sum_{k=-\infty}^\infty2^{k\alpha p}\|f\|_{(E^{q}_{r})_t(S_k)}^p\right]^\frac1{p}
=\|T\|_{(\dot KE_{q',r'}^{-\alpha,p'})_t(\rn)}\|f\|_{(\dot KE_{q,r}^{\alpha,p})_t(\rn)}.
$$
This shows that $T\in((\dot KE_{q,r}^{\alpha,p})_t(\rn))^*$.

Let $T\in ((\dot KE_{q,r}^{\alpha,p})_t(\rn))^*$. For any $k\in\zz$ and $f_k\in (E_r^q)_t(\rn)$, we define
$\widetilde{f_k}:= f_k\mathbf1_{S_k}$, then $\widetilde f_k\in (\dot KE_{q,r}^{\alpha,p})_t(\rn)$ and
$\lf\|\widetilde f_k\r\|_{(\dot KE_{q,r}^{\alpha,p})_t(\rn)}=2^{k\alpha}\|f_k\|_{(E_r^q)_t(\rn)}$.
Let $(T_k,f_k):=(T,\widetilde f_k)$. It is easy to check that $T_k\in (E_{r'}^{q'})_t(\rn)$ and $\lf\|T_k\r\|_{(E_{r'}^{q'})_t(\rn)}\leq2^{k\alpha}\|T\|_{((\dot KE_{q,r}^{\alpha,p})_t(\rn))^*}$.
Now, for any $f\in(\dot KE_{q,r}^{\alpha,p})_t(\rn)$, we have $f\mathbf1_{S_k}\in (E_r^q)_t(\rn)$.
Then, for any given $N,M\in\mathbb N$,
$$
\sum_{k=-N}^M(T_k,f\mathbf1_{S_k})=\left(T,\sum_{k=-N}^Mf\mathbf1_{S_k}\right).
$$
For any $k\in\zz$, choose $f_k'\in(E_r^q)_t(\rn)$ with $\supp (f_k)\subset S_k$ and $\lf\|f_k'\r\|_{(E_r^q)_t(\rn)}=2^{-k\alpha}$ such that
$$
\left(T_k,f_k'\r)\geq 2^{-k\alpha}\lf\|T_k\r\|_{(E_{r'}^{q'})_t(\rn)}-\varepsilon_k,
$$
where $\varepsilon_k>0$ is determined later. Let
$$
f_k:=\left(2^{-k\alpha}\|T_k\|_{(E_{r'}^{q'})_t(\rn)}\right)^{p'-1}f_k'.
$$
For any given $\varepsilon\in(0,\infty)$, select $\varepsilon_k>0$ small enough such that
$$
(T_k,f_k)+ 2^{-|k|}\varepsilon\geq2^{-k\alpha p'}\|T_k\|_{(E_{r'}^{q'})_t(\rn)}^{p'}=2^{k\alpha p}\lf\|f_k\r\|_{(E_r^q)_t(\rn)}^p.
$$
Then we obtain that
\begin{align*}
\sum_{k=-N}^M2^{-k\alpha p'}\lf\|T_k\mathbf1_{S_k}\r\|_{(E_{r'}^{q'})_t(\rn)}^{p'}&\leq 4\varepsilon+\left(T,\sum_{k=-N}^Mf_k\mathbf1_{S_k}\right)\\
&\leq 4\varepsilon+\|T\|_{((\dot KE_{q,r}^{\alpha,p})_t(\rn))^*}\left\|\sum_{k=-N}^Mf_k\right\|_{(\dot KE_{q,r}^{\alpha,p})_t(\rn)}\\
&\leq 4\varepsilon+\|T\|_{((\dot KE_{q,r}^{\alpha,p})_t(\rn))^*}\lf(\sum_{k=-N}^M2^{k\alpha p}\lf\|f_k\mathbf1_{S_k}\r\|_{(E_r^q)_t(\rn)}^p\r)^\frac1p\\
&=4\varepsilon+\|T\|_{((\dot KE_{q,r}^{\alpha,p})_t(\rn))^*}\lf(\sum_{k=-N}^M2^{-k\alpha p'}\lf\|T_k\mathbf1_{S_k}\r\|_{(E_{r'}^{q'})_t(\rn)}^{p'}\r)^\frac1p.
\end{align*}
Letting $\varepsilon\rightarrow0$ and $N,M\rightarrow\infty$, we have
\begin{equation}\label{Eqri}
\left[\sum_{k=-\infty}^\infty2^{-k\alpha p'}\lf\|T_k\mathbf1_{S_k}\r\|_{(E_{r'}^{q'})_t(\rn)}^{p'}\right]^\frac1{p'}\leq \lf\|T\r\|_{((\dot KE_{q,r}^{\alpha,p})_t(\rn))^*}.
\end{equation}
Then we define
$$
\widetilde T(x):=\sum_{k=-\infty}^\infty T_k(x)\mathbf1_{S_k}(x).
$$
It follows from \eqref{Eqri} that $\widetilde T\in (\dot KE_{q',r'}^{-\alpha,p'})_t(\rn)$. Moreover, for any $f\in(\dot KE_{q,r}^{\alpha,p})_t(\rn)$,
it can see that
$$
\lf(\widetilde T,f\r)=\sum_{k=-\infty}^\infty T_k(x)\mathbf1_{S_k}(x)f(x)\,dx
=\sum_{k=-\infty}^\infty\int_{S_k}T_k(x)f(x)\,dx=\sum_{k=-\infty}^\infty(T,f\mathbf1_{S_k})=(T,f).
$$
This shows that $T$ is also in $(\dot KE_{q',r'}^{-\alpha,p'})_t(\rn)$. Hence we prove the result for the case $1<p<\infty$.
For the case $0<p\leq1$, the proof is similar and omitted.
\end{proof}
By the closed-graph theorem, we easily get the following corollary.
\begin{corollary}
Let $\alpha\in\rr$, $t\in(0,\infty)$, $p,q\in[1,\infty)$, $r\in(1,\infty)$ and $1/p+1/p'=1/r+1/r'=1/q+1/q'=1$.
Then $f\in(\dot KE_{q,r}^{\alpha,p})_t(\rn)$ if and only if
$$
\int_\rn f(x)g(x)\,dx<\fz,
$$
where $g\in (\dot KE_{q',r'}^{-\alpha,p'})_t(\rn)$, and
$$
\lf\|f\r\|_{(\dot KE_{q,r}^{\alpha,p})_t(\rn)}=
\sup \lf\{\int_\rn f(x)g(x)\,dx:\lf\|g\r\|_{(\dot KE_{q',r'}^{-\az,p'})_t(\rn)}\leq 1\r\}.
$$
\end{corollary}
A similar result for the non-homogeneous Herz-slice space is given as follows, and hence the details are omitted.
\begin{theorem}\label{oKDuiO}
Let $\alpha\in\rr$, $t,\,p\in(0,\infty)$, $r\in(1,\infty)$ and $q\in[1,\infty)$. The dual space of
$(KE_{q,r}^{\alpha,p})_t(\rn)$ is
$$
\left(\left(KE_{q,r}^{\alpha,p}\right)_t(\rn)\right)^*=\begin{cases}
\left(KE_{q',r'}^{-\alpha,p'}\right)_t(\rn),\quad 1<p<\infty;\\
\left(KE_{q',r'}^{-\alpha,\infty}\right)_t(\rn),\quad 0<p\leq1.
\end{cases}
$$
Furthermore, if $\alpha\in\rr$, $t\in(0,\infty)$, $p,q\in[1,\infty)$, $r\in(1,\infty)$ and $1/p+1/p'=1/r+1/r'=1/q+1/q'=1$.
Then $f\in(KE_{q,r}^{\alpha,p})_t(\rn)$ if and only if
$$
\lf|\int_\rn f(x)g(x)\,dx\r|<\fz,
$$
where $g\in (KE_{q',r'}^{-\alpha,p'})_t(\rn)$, and
$$
\lf\|f\r\|_{(KE_{q,r}^{\alpha,p})_t(\rn)}=
\sup \lf\{\lf|\int_\rn f(x)g(x)\,dx\r|:\lf\|g\r\|_{(KE_{q',r'}^{-\az,p'})_t(\rn)}\leq 1\r\}.
$$
\end{theorem}
\section{Maximal functions on Herz-slice spaces}\label{M}
In this section, we obtain the bounds for the Hardy--Littlewood maximal operator over Herz-slice spaces.
We begin with the definition of the Hardy--Littlewood maximal operator.

For a locllay integrable function $f$, the Hardy--Littlewood maximal operator is defined by setting,
for almost every $x\in\rn$,
$$
Mf(x):=\sup_{r>0}\frac1{|B(x,r)|}\int_{B(x,r)}|f(y)|\,dy.
$$
\begin{theorem}\label{ThMf}
Let $t,\,p\in(0,\infty)$, $q\in[1,\infty)$, $r\in(1,\infty)$ and $-n/q<\alpha<n(1-1/q)$.
Then the Hardy--Littlewood maximal function is bounded on $(\dot KE_{q,r}^{\alpha,p})_t(\rn)$.
\end{theorem}
\begin{lemma}\cite{AP,LWZ}\label{ProMf}
Let $t\in(0,\fz)$, $r\in (1,\fz)$.

(1) If $1<q<\fz$, the Hardy--Littlewood maximal function $M$ is bounded on $(E_r^q)_t(\rn)$.

(2) If $q=1$, $M$ is bounded from $(E_r^q)_t(\rn)$ to $W(E_r^q)_t(\rn)$.
\end{lemma}
\begin{lemma}\label{CF}
Let $t\in(0,\fz)$, $r,\,q\in(1,\fz)$, for all $x_0\in\rn$, $1<r_0<\fz$, a characteristic function
on $B(x_0,r_0)$ satisfies
$$
\lf\|\mathbf1_{B(x_0,r_0)}\r\|_{(E_r^q)_t(\rn)}\leq Cr_0^{n/q}.
$$
\end{lemma}
\begin{proof}
For any $x\in\rn$, we have
\begin{align*}
\lf\|\mathbf1_{B(x_0,r_0)}\r\|_{(E_r^q)_t(\rn)}
&=\lf\{\int_{\rn}\lf[\frac{\lf\|\mathbf1_{B(x_0,r_0)}\mathbf1_{B(x,t)}\r\|_{L^r}}
{\lf\|\mathbf1_{B(x,t)}\r\|_{L^r}}\r]^qdx\r\}^\frac1q\\
&=\frac1{\lf\|\mathbf1_{B(\vec{0},t)}\r\|_{L^r(\rn)}}\lf[\int_{\rn}
\lf\|\mathbf1_{B(\vec{0},r_0)}\mathbf1_{B(x-x_0,t)}\r\|_{L^r}^qdx\r]^\frac1q\\
&=\frac1{\lf\|\mathbf1_{B(\vec{0},t)}\r\|_{L^r(\rn)}}\lf[\int_{\rn}
\lf\|\mathbf1_{B(\vec{0},r_0)}\mathbf1_{B(x,t)}\r\|_{L^r}^qdx\r]^\frac1q\\
&=\lf\|\mathbf1_{B(\vec{0},r_0)}\r\|_{(E_r^q)_t(\rn)}.
\end{align*}
By this, without loss of generality, suppose that $B_1:=B(\vec{0},1)$ and $B_2:=B(\vec{0},r_0)$ with $1<r_0<\fz$.
 By the geometric property, we know that there exist $M\in\nn$ with $M\sim |B_2|^n$ and
 $\{x_1,...,x_M\}$ such that $B(\vec{0},r_0)\subset\bigcup_{i=1}^MB(x_i,1)$,  which implies that
\begin{align*}
\lf\|\mathbf1_{B(\vec{0},r_0)}\r\|_{(E_r^q)_t(\rn)}&=\lf\|\mathbf1_{B_2}\r\|^\frac1q_{(E_{r/q}^1)_t(\rn)}\\
&\le\lf\|\sum_{i=1}^M\mathbf1_{B(x_i,1)}\r\|^\frac1q_{(E_{r/q}^1)_t(\rn)}
\lesssim \lf(\sum_{i=1}^M\lf\|\mathbf1_{B(x_i,1)}\r\|_{(E_{r/q}^1)_t(\rn)}\r)^\frac1q\\
&\sim |B_2|^\frac1q\lf\|\mathbf1_{B(\vec{0},1)}\r\|_{(E_r^q)_t(\rn)}
\sim r_0^{n/q}.
\end{align*}
This completes the proof.
\end{proof}
\begin{remark}
Let $t\in(0,\fz)$, $r,\,q\in(1,\fz)$, $k\in\zz$, a characteristic function
on $S_k$ satisfies
$$
\lf\|\mathbf1_{S_k}\r\|_{(E_r^q)_t(\rn)}\leq \lf\|\mathbf1_{B_k}\r\|_{(E_r^q)_t(\rn)}\leq C2^{kn/q}.
$$
\end{remark}
\begin{proof}[Proof of Theorem \ref{ThMf}]
Write
$$
f(x)=\sum\limits_{l\in\zz} f(x)\mathbf1_{S_l}(x):= \sum\limits_{l\in\zz} f_l(x).
$$
For $q\in(1,\fz)$, we get
\begin{align*}
\lf\|Mf\r\|_{(\dot KE_{q,r}^{\alpha,p})_t(\rn)}&=\lf(\sum\limits_{k\in\zz}2^{k\alpha p}\lf\|Mf\mathbf1_{S_k}\r\|^p_{(E_r^q)_t(\rn)}\r)^\frac1p
=\lf(\sum\limits_{k\in\zz}2^{k\alpha p}\lf\|\sum\limits_{l=-\fz}^\fz Mf_l\mathbf1_{S_k}\r\|^p_{(E_r^q)_t(\rn)}\r)^\frac1p\\
&\leq \lf(\sum\limits_{k\in\zz}2^{k\alpha p}\lf\|\sum\limits_{l=-\fz}^{k-2} Mf_l\mathbf1_{S_k}\r\|^p_{(E_r^q)_t(\rn)}\r)^\frac1p
+\lf(\sum\limits_{k\in\zz}2^{k\alpha p}\lf\|\sum\limits_{l=k-1}^{k+1} Mf_l\mathbf1_{S_k}\r\|^p_{(E_r^q)_t(\rn)}\r)^\frac1p\\
&+\lf(\sum\limits_{k\in\zz}2^{k\alpha p}\lf\|\sum\limits_{l=k+2}^\fz Mf_l\mathbf1_{S_k}\r\|^p_{(E_r^q)_t(\rn)}\r)^\frac1p
:=I+II+III.
\end{align*}
For $p\in(0,1]$, applying the fact that for $l\leq k-2$ and $x\in S_k$,
$
\lf|M(f_l)(x)\r|\leq C2^{-kn}\lf\|f_l\r\|_{L^1(\rn)},
$
and using Lemmas \ref{Holder} and \ref{CF} one can deduce that
\begin{align*}
I&=\lf(\sum\limits_{k\in\zz}2^{k\alpha p}\lf\|\sum\limits_{l=-\fz}^{k-2} Mf_l\mathbf1_{S_k}\r\|^p_{(E_r^q)_t(\rn)}\r)^\frac1p\\\
&\lesssim\lf(\sum\limits_{k\in\zz}2^{k\alpha p}\lf\|\sum\limits_{l=-\fz}^{k-2} \lf\|f_l\r\|_{L^1(\rn)}2^{-kn}\mathbf1_{S_k}\r\|^p_{(E_r^q)_t(\rn)}\r)^\frac1p\\
&\lesssim\lf(\sum\limits_{k\in\zz}2^{k\alpha p}\lf\|\sum\limits_{l=-\fz}^{k-2} \lf\|f_l\r\|_{(E_r^q)_t(\rn)}\lf\|\mathbf1_{S_l}\r\|_{(E_{r'}^{q'})_t(\rn)}
2^{-kn}\mathbf1_{S_k}\r\|^p_{(E_r^q)_t(\rn)}\r)^\frac1p\\
&\lesssim \lf(\sum\limits_{k\in\zz}2^{k\alpha p}\sum\limits_{l=-\fz}^{k-2} \lf\|f_l\r\|^p_{(E_r^q)_t(\rn)}\lf\|\mathbf1_{S_l}\r\|^p_{(E_{r'}^{q'})_t(\rn)}2^{-knp}
\lf\|\mathbf1_{S_k}\r\|^p_{(E_r^q)_t(\rn)}\r)^\frac1p\\
&\lesssim \lf(\sum\limits_{k\in\zz}2^{k\alpha p}\sum\limits_{l=-\fz}^{k-2} \lf\|f_l\r\|^p_{(E_r^q)_t(\rn)}\lf\|\mathbf1_{B_l}\r\|^p_{(E_{r'}^{q'})_t(\rn)}2^{-knp}
\lf\|\mathbf1_{B_k}\r\|^p_{(E_r^q)_t(\rn)}\r)^\frac1p\\
\end{align*}
\begin{align*}
&\lesssim\lf(\sum\limits_{k\in\zz}2^{k\alpha p}\sum\limits_{l=-\fz}^{k-2} \|f_l\|^p_{(E_r^q)_t(\rn)}
2^{(k-l)p(1/q-1)n}\r)^\frac1p\\
&\lesssim \lf(\sum\limits_{k\in\zz}2^{k\alpha p}\sum\limits_{l=-\fz}^{k-2} \|f_l\|^p_{(E_r^q)_t(\rn)}
2^{-(k-l)p\az}\r)^\frac1p\\
&\lesssim \lf(\sum\limits_{l\in\zz}2^{l\az p}\|f_l\|^p_{(E_r^q)_t(\rn)}\r)^\frac1p
\lesssim \|f\|_{(\dot KE_{q,r}^{\alpha,p})_t(\rn)},
\end{align*}
and for $p\in(1,\fz)$,
\begin{align*}
I&=\lf(\sum\limits_{k\in\zz}2^{k\alpha p}\lf\|\sum\limits_{l=-\fz}^{k-2} Mf_l\mathbf1_{S_k}\r\|^p_{(E_r^q)_t(\rn)}\r)^\frac1p\\\
&\lesssim\lf(\sum\limits_{k\in\zz}2^{k\alpha p}\lf\|\sum\limits_{l=-\fz}^{k-2} \|f_l\|_{L^1(\rn)}2^{-kn}\mathbf1_{S_k}\r\|^p_{(E_r^q)_t(\rn)}\r)^\frac1p\\
&\lesssim\lf(\sum\limits_{k\in\zz}2^{k\alpha p}\lf\|\sum\limits_{l=-\fz}^{k-2} \|f_l\|_{(E_r^q)_t(\rn)}\lf\|\mathbf1_{S_l}\r\|_{(E_{r'}^{q'})_t(\rn)}
2^{-kn}\mathbf1_{S_k}\r\|^p_{(E_r^q)_t(\rn)}\r)^\frac1p\\
&\lesssim \lf(\sum\limits_{k\in\zz}2^{k\alpha p}\lf(\sum\limits_{l=-\fz}^{k-2} \|f_l\|_{(E_r^q)_t(\rn)}\r)^p\lf\|\mathbf1_{S_l}\r\|^p_{(E_{r'}^{q'})_t(\rn)}2^{-knp}
\lf\|\mathbf1_{S_k}\r\|^p_{(E_r^q)_t(\rn)}\r)^\frac1p\\
&\lesssim \lf(\sum\limits_{k\in\zz}2^{k\alpha p}\lf(\sum\limits_{l=-\fz}^{k-2} \|f_l\|_{(E_r^q)_t(\rn)}\r)^p\lf\|\mathbf1_{B_l}\r\|^p_{(E_{r'}^{q'})_t(\rn)}2^{-knp}
\lf\|\mathbf1_{B_k}\r\|^p_{(E_r^q)_t(\rn)}\r)^\frac1p\\
&\lesssim \lf(\sum\limits_{k\in\zz}2^{k\alpha p}\lf(\sum\limits_{l=-\fz}^{k-2} \|f_l\|_{(E_r^q)_t(\rn)}\r)^p2^{(k-l)p(1/q-1)n}\r)^\frac1p\\
&\lesssim \lf(\sum\limits_{k\in\zz}2^{k\alpha p}\lf(\sum\limits_{l=-\fz}^{k-2} \|f_l\|_{(E_r^q)_t(\rn)}2^{(k-l)(1/q-1)n}\r)^p\r)^\frac1p\\
&\lesssim \lf(\sum\limits_{k\in\zz}2^{k\alpha p}\lf(\sum\limits_{l=-\fz}^{k-2} \|f_l\|^p_{(E_r^q)_t(\rn)}2^{(k-l)(1/q-1)np/2}\r)
\lf(\sum\limits_{l=-\fz}^{k-2}2^{(k-l)(1/q-1)np'/2}\r)^{p/p'}\r)^\frac1p\\
&\lesssim \|f\|_{(\dot KE_{q,r}^{\alpha,p})_t(\rn)}.
\end{align*}
Using Lemma \ref{ProMf}, we obtain the following estimate for $II$:
\begin{align*}
II
&\lesssim \lf(\sum\limits_{k\in\zz}2^{k\alpha p}\lf\|\sum\limits_{l=k-1}^{k+1} Mf_l\r\|^p_{(E_r^q)_t(\rn)}\r)^\frac1p
\lesssim \lf(\sum\limits_{k\in\zz}\sum\limits_{l=k-1}^{k+1}2^{(k-l)\alpha p}2^{l\az p}\|f_l\|^p_{(E_r^q)_t(\rn)}\r)^\frac1p\\
&\lesssim \lf(\sum\limits_{l\in\zz}2^{l\az p}\|f_l\|^p_{(E_r^q)_t(\rn)}\r)^\frac1p
\lesssim \|f\|_{(\dot KE_{q,r}^{\alpha,p})_t(\rn)}.
\end{align*}
We then prove $III$. Note that if $l\ge k+2$, then
$$
\lf|M(f_l)(x)\r|\mathbf1_{S_k}\leq C2^{-nl}\lf\|f_l\r\|_{L^1(\rn)}.
$$
For $0<p\leq 1$,
(\ref{cp}) and Lemma \ref{Holder} yield
\begin{align*}
III
&\lesssim \lf(\sum\limits_{k\in\zz}2^{k\alpha p}\lf\|\sum\limits_{l=k+2}^\fz Mf_l\mathbf1_{S_k}\r\|^p_{(E_r^q)_t(\rn)}\r)^\frac1p\\
&\lesssim \lf(\sum\limits_{k\in\zz}2^{k\alpha p}\lf\|\sum\limits_{l=k+2}^\fz \|f_l\|_{L^1(\rn)}2^{-nl}\mathbf1_{S_k}\r\|^p_{(E_r^q)_t(\rn)}\r)^\frac1p\\
&\lesssim \lf(\sum\limits_{k\in\zz}2^{k\alpha p}\lf\|\sum\limits_{l=k+2}^\fz \|f_l\|_{(E_r^q)_t(\rn)}\lf\|\mathbf1_{S_l}\r\|_{(E_{r'}^{q'})_t(\rn)}2^{-ln}
\mathbf1_{S_k}\r\|^p_{(E_r^q)_t(\rn)}\r)^\frac1p\\
&\lesssim \lf(\sum\limits_{k\in\zz}2^{k\alpha p}\sum\limits_{l=k+2}^\fz\|f_l\|^p_{(E_r^q)_t(\rn)}
\lf\|\mathbf1_{B_l}\r\|^p_{(E_{r'}^{q'})_t(\rn)}2^{-lpn}
\lf\|\mathbf1_{B_k}\r\|^p_{(E_r^q)_t(\rn)}\r)^\frac1p\\
&\lesssim \lf(\sum\limits_{k\in\zz}2^{k\alpha p}\sum\limits_{l=k+2}^\fz\|f_l\|^p_{(E_r^q)_t(\rn)}
2^{lpn/q'}2^{-lpn}2^{knp/q}\r)^\frac1p\\
&\lesssim \lf(\sum\limits_{k\in\zz}2^{k\alpha p}\sum\limits_{l=k+2}^\fz\|f_l\|^p_{(E_r^q)_t(\rn)}
2^{(l-k)p\az}\r)^\frac1p\\
&\lesssim \lf(\sum\limits_{k\in\zz} 2^{l\az p} \|f_l\|^p_{(E_r^q)_t(\rn)}\r)^\frac1p
\lesssim \|f\|_{(\dot KE_{q,r}^{\alpha,p})_t(\rn)},
\end{align*}
and by Lemma \ref{Holder} and H\"older's inequality for $1<p<\fz$,
\begin{align*}
III
&\lesssim \lf(\sum\limits_{k\in\zz}2^{k\alpha p}\lf\|\sum\limits_{l=k+2}^\fz Mf_l\mathbf1_{S_k}\r\|^p_{(E_r^q)_t(\rn)}\r)^\frac1p\\
&\lesssim \lf(\sum\limits_{k\in\zz}2^{k\alpha p}\lf\|\sum\limits_{l=k+2}^\fz \|f_l\|_{L^1(\rn)}2^{-ln}\mathbf1_{S_k}\r\|^p_{(E_r^q)_t(\rn)}\r)^\frac1p\\
&\lesssim \lf(\sum\limits_{k\in\zz}2^{k\alpha p}\lf\|\sum\limits_{l=k+2}^\fz \|f_l\|_{(E_r^q)_t(\rn)}\lf\|\mathbf1_{S_l}\r\|_{(E_{r'}^{q'})_t(\rn)}2^{-nl}
\mathbf1_{S_k}\r\|^p_{(E_r^q)_t(\rn)}\r)^\frac1p\\
&\lesssim \lf(\sum\limits_{k\in\zz}2^{k\alpha p}\lf(\sum\limits_{l=k+2}^\fz\|f_l\|_{(E_r^q)_t(\rn)}\r)^p
\lf\|\mathbf1_{B_l}\r\|^p_{(E_{r'}^{q'})_t(\rn)}2^{-lpn}
\lf\|\mathbf1_{B_k}\r\|^p_{(E_r^q)_t(\rn)}\r)^\frac1p\\
&\lesssim \lf(\sum\limits_{k\in\zz}2^{k\alpha p}\sum\limits_{l=k+2}^\fz\|f_l\|^p_{(E_r^q)_t(\rn)}
2^{(k-l)np/(2q)}\lf(\sum\limits_{l=k+2}^\fz2^{(k-l)np'/(2q)}\r)^{p/p'}\r)^\frac1p\\
&\lesssim \lf(\sum\limits_{k\in\zz} 2^{l\az p} \|f_l\|^p_{(E_r^q)_t(\rn)}\lf(\sum\limits_{k=-\fz}^{l-2}2^{(k-l)np/(2q)}\r)\r)^\frac1p
\lesssim \|f\|_{(\dot KE_{q,r}^{\alpha,p})_t(\rn)}.
\end{align*}
For $q=1$, for any $\lz>0$,
\begin{align*}
\lz\lf(\sum_{k\in\zz}2^{k\az p}\lf\|\mathbf 1_{\{x\in S_k:~|Mf(x)|>\lambda\}}\r\|^p_{(E_r^q)_t(\rn)}\r)^\frac1p
&\leq
\lz\lf(\sum_{k\in\zz}2^{k\az p}\lf\|\mathbf 1_{\lf\{x\in S_k:~\lf| M\lf(\sum\limits_{l=-\fz}^{k-3} f_l\r)(x)\r|>\lambda/3\r\}}\r\|^p_{(E_r^q)_t(\rn)}\r)^\frac1p\\
&+\lz\lf(\sum_{k\in\zz}2^{k\az p}\lf\|\mathbf 1_{\lf\{x\in S_k:~\lf| M\lf(\sum\limits_{l=k-2}^{k+2} f_l\r)(x)\r|>\lambda/3\r\}}\r\|^p_{(E_r^q)_t(\rn)}\r)^\frac1p\\
&+\lz\lf(\sum_{k\in\zz}2^{k\az p}\lf\|\mathbf 1_{\lf\{x\in S_k:~\lf| M\lf(\sum\limits_{l=k+3}^{\fz} f_l\r)(x)\r|>\lambda/3\r\}}\r\|^p_{(E_r^q)_t(\rn)}\r)^\frac1p\\
&:=I+II+III.
\end{align*}
Since $M$ is bounded from $(E_r^1)_t(\rn)$ to $W(E_r^1)_t(\rn)$ in Lemma \ref{ProMf}, then
$$
II\lesssim\lf(\sum_{k\in\zz}2^{k\az p}\sum\limits_{l=k-2}^{k+2}\lf\|f_l\mathbf1_{S_l}\r\|^p_{(E_r^1)_t(\rn)}\r)^\frac1p
\lesssim \|f\|_{(\dot KE_{1,r}^{\alpha,p})_t(\rn)}.
$$
$I$ and $III$ hold since the similar proof of the case of $q\in(1,\fz)$.

This completes the proof of Theorem \ref{ThMf}.
\end{proof}
The $(KE_{q,r}^{\alpha,p})_t(\rn)$ boundedness of the Hardy--Littlewood maximal function $M$
is valid in the following, via a similar argument to that given in the proof of Theorem \ref{ThMf},
\begin{theorem}
Let $t,\,p\in(0,\infty)$, $q\in[1,\infty)$, $r\in(1,\infty)$ and $-n/q<\alpha<n(1-1/q)$.
Then the Hardy--Littlewood maximal function is bounded on $(KE_{q,r}^{\alpha,p})_t(\rn)$.
\end{theorem}

%
%
%
%
%



\hspace*{-0.6cm}\textbf{\bf Competing interests}\\
The authors declare that they have no competing interests.\\

\hspace*{-0.6cm}\textbf{\bf Funding}\\
The research was supported by the National Natural Science Foundation of China (Grant No. 12061069) and the Natural
Science Foundation Project of Chongqing, China (Grant No. cstc2021jcyj-msxmX0705).\\

\hspace*{-0.6cm}\textbf{\bf Authors contributions}\\
All authors contributed equality and significantly in writing this paper. All authors read and approved the final manuscript.\\

\hspace*{-0.6cm}\textbf{\bf Acknowledgments}\\
All authors would like to express their thanks to the referees for valuable advice regarding previous version of this paper.\\

\hspace*{-0.6cm}\textbf{\bf Authors detaials}\\
Yuan Lu(Luyuan\_y@163.com) and Jiang Zhou(zhoujiang@xju.edu.cn),
College of Mathematics and System Science, Xinjiang University, Urumqi, 830046, P.R China.\\
Songbai Wang(haiyansongbai@163.com),
College of Mathematics and Statistics, Chongqing Three Gorges University, Chongqing 404130, P.R China.\\

\bigskip
\noindent Yuan Lu\\
College of Mathematics and System Sciences\\
Xinjiang University\\
Urumqi 830046, China\\
\smallskip
\noindent\emph{Email address}: \texttt{Luyuan\_y@163.com}\\

\noindent Jiang Zhou\\
College of Mathematics and System Sciences\\
Xinjiang University\\
Urumqi 830046, China\\
\smallskip
\noindent\emph{Email address}: \texttt{zhoujiang@xju.edu.cn}\\

\noindent Songbai Wang\\
Chongqing Three Gorges University\\
Chongqing 404130, China\\
\smallskip
\noindent\emph{Email address}: {haiyansongbai@163.com}\\
\bigskip \medskip

\end{document}